\documentclass{amsart}
\usepackage{graphicx} 
\RequirePackage{hyperref}
\usepackage{tikz}
\usetikzlibrary{patterns}
\usepackage[color,matrix,arrow]{xy}
\usepackage{enumerate}

\newtheorem{theorem}{Theorem}[section]
\newtheorem{lemma}[theorem]{Lemma}
\newtheorem{corollary}[theorem]{Corollary}

\newtheorem{fact}[theorem]{Fact}
\newtheorem{question}[theorem]{Question}

\newtheorem{theoremintro}{Theorem}

\theoremstyle{definition}
\newtheorem{definition}[theorem]{Definition}
\newtheorem{example}[theorem]{Example}
\newtheorem{convention}[theorem]{Convention}

\theoremstyle{remark}
\newtheorem{remark}[theorem]{Remark}

\numberwithin{equation}{section}

\DeclareMathOperator{\diam}{diam}
\DeclareMathOperator{\Sym}{Sym}
\DeclareMathOperator{\dis}{dis}
\DeclareMathOperator{\sep}{sep}
\DeclareMathOperator{\supp}{supp}

\DeclareMathOperator{\graph}{Gr}

\newcommand{\N}{\mathbb{N}}
\newcommand{\R}{\mathbb{R}}
\newcommand{\GH}{d_{GH}}
\newcommand{\card}[1]{\# {#1}}
\newcommand{\sint}[1]{#1^\circ}
\newcommand{\GHcpt}{\mathcal{M}}
\newcommand{\GHeq}[1]{\mathcal{M}_{[#1]}}
\newcommand{\GHleq}[1]{\mathcal{M}_{\leq #1}}
\newcommand{\GHfin}{\mathcal{M}_{<\omega}}
\newcommand{\X}{\mathcal{X}}
\newcommand{\Xeq}[1]{\mathcal{X}_{[#1]}}
\newcommand{\Xleq}[1]{\mathcal{X}_{\leq #1}}
\newcommand{\Xfin}{\mathcal{X}_{<\omega}}
\newcommand{\sym}[1]{\Sym(#1)}

\title[The Gromov-Hausdorff and Gromov-Prokhorov spaces]{Topological dimension of the Gromov-Hausdorff and Gromov-Prokhorov spaces}

\author[H. Nakajima]{Hiroki Nakajima}
\address{Graduate School of Science and Engineering, Ehime University, Matsuyama, 790-8577, Japan}
\email{nakajima.hiroki.nz@ehime-u.ac.jp}

\author[T. Yamauchi]{Takamitsu Yamauchi}
\address{Graduate School of Science and Engineering, Ehime University, Matsuyama, 790-8577, Japan}
\email{yamauchi.takamitsu.ts@ehime-u.ac.jp}

\author[N. Zava]{Nicol\`o Zava}
\address{Institute of Science and Technology Austria (ISTA), 3400 Klosterneuburg, Austria}
\email{nicolo.zava@gmail.com}



\thanks{The first and second named authors were supported by JSPS KAKENHI Grant Numbers 22K13908 and 24K06739, respectively. The third named author was supported by the FWF Grant, Project number I4245-N35.}

\begin{document}

\maketitle

\begin{abstract}
The Gromov-Hausdorff distance is a dissimilarity metric capturing how far two spaces are from being isometric. The Gromov-Prokhorov distance is a similar notion for metric measure spaces. In this paper, we study the topological dimension of the Gromov-Hausdorff and Gromov-Prokhorov spaces. We show that the dimension of the space of isometry classes of metric spaces with at most $n$ points endowed with the Gromov-Hausdorff distance is $\frac{n(n-1)}{2}$, and that of mm-isomorphism classes of metric measure spaces whose support consists of $n$ points is $\frac{(n+2)(n-1)}{2}$. Hence, the spaces of all isometry classes of finite metric spaces and of all mm-isomorphism classes of finite metric measure spaces are strongly countable dimensional. If, instead, the cardinalities are not limited, the spaces are strongly infinite-dimensional.
\end{abstract}

\section{Introduction}

In geometry and topology, and indeed throughout mathematics, a fundamental problem is detecting whether two structures are the same up to isomorphism. A subsequent intuitive question is the following: how far are two structures from being isomorphic? A quantitative, more nuanced answer enriches our understanding of the mathematical structures considered. A prominent example can be found in the study of metric spaces. Let $X$ and $Y$ be two metric spaces embedded into a common ambient metric space $(Z,d)$. The first option to compare them is the {\em Hausdorff distance}. It is defined as the infimum $\varepsilon>0$ such that, for every pair of points $x\in X$ and $y\in Y$, there are $y_x\in Y$ and $x_y\in X$ satisfying $d(x,y_x)<\varepsilon$ and $d(y,x_y)<\varepsilon$. In \cite{gromov:81}, Gromov removed the dependency on the ambient space and considered the distance between two metric spaces $X$ and $Y$ to be the infimum Hausdorff distance of isometric copies of $X$ and $Y$ into a common metric space. For the naturality of the approach, he still called it Hausdorff distance, whereas it is now known as {\em Gromov-Hausdorff distance} as a recognition of the important applications he derived with it. Let us mention that similar concepts can be already found in the earlier papers \cite{edwards:75} and \cite{kadets:75}. 

The theory of metric measure spaces followed a similar developmental path. Given a complete separable metric space, several metrics have been introduced and studied to compare two Borel probability measures on it. We refer to \cite{villani:03} for a wide introduction. Among these distances we can mention, for example, the Wasserstein distances and the Prokhorov distance (also known as L\'evy-Prokhorov distance). The latter is defined as follows: given two Borel probability measures $\mu$ and $\nu$ on a complete separable metric space $X$, their {\em Prokhorov distance} is
$$d_{Pr}(\mu,\nu)=\inf\{\varepsilon>0\mid\forall A\in\mathfrak B(X),\,\mu(A)\leq\nu(A^\varepsilon)+\varepsilon\},$$
where $A^\varepsilon=\bigcup_{a\in A}B(a,\varepsilon)$. 
Again with an initial contribution of Gromov (\cite{gromov:81}), some of these distances have been extended to compare metric measure spaces, i.e., complete separable metric spaces with Borel probability measures. In particular, let us recall the {\em Gromov-Prokhorov distance}, defined as the infimum of the Prokhorov distances achievable by isometrically embedding the two metric measure spaces into a common complete separable metric space (the probability measures compared are the push-forwards induced by the two isometric embeddings). It was proved in \cite{loehr:13} that the Gromov-Prokhorov distance is bi-Lipschitz equivalent to another notion introduced by Gromov (\cite{gromov:81}), the box distance. In particular, the Gromov-Prokhorov distance and the box distance share the same topological properties. Another distance between metric measure spaces, which has been played a crucial role in the computational topology since its introduction in \cite{memoli:11}, is the Gromov-Wasserstein distance.

In this paper, we focus on the space $\GHcpt$ of isometry classes of non-empty compact metric spaces endowed with the Gromov-Hausdorff distance, sometimes called the {\em Gromov-Hausdorff space}, and on the space $\X$ of mm-isomorphism classes of metric measure spaces equipped with the box distance and their subspaces. Let us recall once again that $\X$ is homeomorphic---actually, bi-Lipschitz equivalent---to the {\em Gromov-Prokhorov space} consisting of mm-isomorphism classes of metric measure spaces equipped with the Gromov-Prokhorov distance. In particular, we investigate the subspaces $\GHfin$, $\Xfin$, $\GHleq{n}$ and $\Xleq{n}$ consisting of objects with finitely many points and with at most $n$ points, respectively. 
To investigate their complexity, let us consider two questions:
\begin{enumerate}[(1)]
	\item What spaces can be embedded into $\GHcpt$ and $\X$ and their subspaces? 
	\item In which spaces can $\GHcpt$ and $\X$ and their subspaces be embedded?
\end{enumerate}
In the intuitive formulation above, we purposely kept the notion of embeddability vague because it can be widely interpreted---e.g., topological embeddings, bi-Lipschitz embeddings, coarse embeddings (\cite{gromov:93})---depending on the specific viewpoint from which the question is investigated.

Let us consider (1). It is known that the Hilbert cube can be topologically embedded into $\GHcpt$ (\cite[Theorem 1.3]{ishiki:22}). It is shown in \cite[Theorem 4.1]{iliadis-ivanov-tuzhilin:17} that it is possible to isometrically embed in $\GHleq{n(n-1)/2}$ arbitrarily large balls of $\R^n$ with the supremum metric. Furthermore, every finite metric space of $n$ points can be coarsely embedded into $\GHleq{n+1}$ (\cite{zava:24}). At the cost of increasing the number of points, in \cite{ostrovska-ostrovskii:25}, the authors provided an isometric embedding of any metric space with at most $n$ points into $\GHleq{2n+1}$. 

To answer question (2) while still increasing the insight around question (1), dimension notions can be exploited. One of the motivations to introduce a topological notion of dimension was to prove that $\R^m$ and $\R^n$ are homeomorphic if and only if $m=n$. We refer to classic monograph \cite{engelking:95} for bibliography and historical notes. Similarly, the Assouad dimension was introduced in \cite{assouad:83} to study the bi-Lipschitz embeddability of metric spaces, in particular fractals, into $\R^n$. In \cite{zava:24}, the author used asymptotic dimension (\cite{gromov:93}) and Assouad dimension to study coarse and bi-Lipschitz embeddability of $\GHfin$ and $\GHleq{n}$ into Hilbert spaces. Furthermore, it is shown there that $\GHfin$ cannot be coarsely embedded into any uniformly convex Banach space. In addition to their intrinsic importance, the motivation for the study conducted in \cite{zava:24} comes from computational topology. We refer to \cite{pritchard-weighill:24} for a detailed presentation of further applications of dimension theory in computational topology.

In this paper, we focus on the topological dimension of $\GHcpt$, $\X$ and their subspaces. Note that they are separable metric spaces (see {\cite[Proposition 43]{petersen:06} and} \cite[Proposition 4.25]{shioya:16} for the separability of $\GHcpt$ and $\X$, respectively), and so the covering, the small inductive and the large inductive dimensions coincide (\cite[Theorem 1.7.7]{engelking:95}).
We prove the following result.
\begin{theoremintro}\label{thm:dimension}
$\dim\GHleq{n}=\frac{n(n-1)}{2}$, and $\dim\Xleq{n}=\frac{(n+2)(n-1)}{2}$.
\end{theoremintro}
Therefore, $\GHleq{n}$ and $\Xleq{n}$ cannot be topologically embedded into any space whose dimension is strictly smaller than $\frac{n(n-1)}{2}$ and $\frac{(n+2)(n-1)}{2}$, respectively. 
As another consequence of Theorem \ref{thm:dimension}, $\GHfin$ and $\Xfin$ are strongly countable dimensional, i.e., each of them can be represented as a union of countably many finite-dimensional closed subspaces. Furthermore, we show the following result, which is similar to the case of $\GHcpt$ due to \cite{ishiki:22}. 
\begin{theoremintro}
\label{thm:HC-X}
The Hilbert cube can be topologically embedded into $\X$.
\end{theoremintro}
Hence both $\GHcpt$ and $\X$ are strongly infinite-dimensional (see \cite[Theorem 1.8.2, Definition 6.1.1 and Proposition 6.1.3]{engelking:95}), and in particular, they are not countable-dimensional, i.e., each of them cannot be represented as a union of countably many finite-dimensional subspaces (\cite[Theorem 6.1.10]{engelking:95}).
Note that $\GHfin$ and $\Xfin$ are dense in $\GHcpt$ and $\X$, respectively (see \cite[Proposition 4.20]{shioya:16} for $\X$).

The paper is organised as follows. In Section \ref{sec:background}, we provide the needed background definitions and results concerning the Gromov-Hausdorff distance (\S\ref{sub:GH}), the box distance (\S\ref{sub:box}) and the covering dimension (\S\ref{sub:dim}). In Sections \ref{sec:dimensions} and \ref{sec:proof-of-theoremB},
we prove Theorems \ref{thm:dimension} and \ref{thm:HC-X}, respectively.

\section{Background and definitions}\label{sec:background}

Let $\mathbb{N}$ denote the set of all strictly positive integers.
Given a metric space $(X,d_X)$, $x \in X$ and $\varepsilon>0$, let us denote by $B_{d_X}(x,\varepsilon)$ the open ball in $(X,d_X)$ around $x$ with radius $\varepsilon$.

\subsection{Gromov-Hausdorff distance}\label{sub:GH}

Let $(X,d_X)$ and $(Y,d_Y)$ be metric spaces.
For a relation $R \subset X\times Y$, the \emph{distortion} of $R$, denoted by $\dis R$,  is defined by
\begin{align*}
\dis R = \sup_{(x_1,y_1), (x_2,y_2)\in R} |d_X (x_1,x_2) - d_Y(y_1,y_2)|.
\end{align*}
A subset $R \subset X\times Y$ is called a \emph{correspondence} between $X$ and $Y$ if $R\cap (\{x\}\times Y) \ne \emptyset \ne R\cap (X \times \{y\})$ for every $(x,y) \in X\times Y$.
Let $\mathcal{R}(X,Y)$ denote the all correspondences between $X$ and $Y$.
Let $d_{GH}$ denote the Gromov-Hausdorff distance (see \cite[Definition 7.3.10]{burago-burago-ivanov:01}).
Then the following holds (see \cite[Theorem 7.3.25]{burago-burago-ivanov:01}):
\begin{align*}
d_{GH} ((X,d_X),(Y,d_Y)) = \frac{1}{2} \inf_{R \in \mathcal{R}(X,Y)} \dis R.
\end{align*}

Let $\GHcpt$ denote the metric space of all isometry classes of compact metric spaces endowed with the Gromov-Hausdorff distance $\GH$ (see \cite[Theorem 7.3.30]{burago-burago-ivanov:01}), and for $n \in \N$, let $\GHeq{n}$ (resp., $\GHleq{n}$, $\GHfin$) denote the metric subspace of $\GHcpt$ consisting of all isometry classes of metric spaces with precisely $n$ points (resp., at most $n$ points, finite points).
We refer to \cite{tuzhilin:20} for properties of $\GHcpt$ and its subspaces.

\begin{convention}
When we say that $X$ is a metric space, $X$ is assumed to admit a metric $d_X$ unless otherwise stated.
We use $[X,d_X]$ (or $[X]$ when there is no confusion) to denote the isometric class of $(X,d_X)$. 
\end{convention} 

\subsection{Metric measure spaces and box distance}\label{sub:box}

By a \emph{metric measure space} $(X,d_X,\mu_X)$, we mean a complete separable metric space $(X,d_X)$ with a Borel probability measure $\mu_X$ on $X$.
The \emph{support} of $\mu_X$, denoted by $\supp\mu_X$, is the set of all $x \in X$ such that $\mu_X(U)>0$ for every open neighborhood $U$ of $x$. 

Let $(X,d_X,\mu_X)$ and $(Y,d_Y,\mu_Y)$ be metric measure spaces.
They are said to be \emph{mm-isomorphic} if there exists an isometry $f \colon \supp \mu_X \to \supp \mu_Y$ such that $f_*\mu_X=\mu_Y$, where $f_*\mu_X$ denote the {\em push-forward measure of $\mu_X$ by $f$} defined by $f_*\mu_X (B)=\mu_X(f^{-1}(B))$ for each Borel subset $B$ of $Y$.

A Borel measure $\pi$ on $X\times Y$ is said to be a \emph{coupling} (or \emph{transport plan}) between $\mu_X$ and $\mu_Y$ if its {\em marginals} are $\mu_X$ and $\mu_Y$, i.e., if $\pi (A\times Y)=\mu_X(A)$ and $\pi(X\times B)=\mu_Y(B)$ for any Borel subsets $A$ and $B$ of $X$ and $Y$, respectively. Let $\Pi(\mu_X,\nu_X)$ denote the set of all couplings between $\mu_X$ and $\mu_Y$.
Let $\square$ denote the box distance (see \cite[Definition 4.4]{shioya:16}).
Let $\mathcal{F}(X\times Y)$ denote the set of all closed subsets of $X\times Y$.
Then the following holds \cite[Theorem 1.1]{nakajima:22}:
\begin{equation}\label{eq:box}
\begin{aligned}
\square ((X,d_X,\mu_X), (Y,d_Y,\mu_Y) ) = \min_{\pi \in \Pi (\mu_X,\mu_Y)} \min_{S \in \mathcal{F}(X\times Y)} \max\{1 -\pi (S), \dis S \}.
\end{aligned}
\end{equation}
While we refer to \cite[Theorem 1.1]{nakajima:22} for the details, we briefly mention that the existence of the minimizer in \eqref{eq:box} follows from the facts that $\Pi(\mu_X,\mu_Y)$ is compact with respect to the weak convergence and that  $\mathcal F(X\times Y)$ is compact with respect to the Kuratowski-Painlev\'{e} convergence.

\begin{convention}
When we say that $X$ is a metric measure space, $X$ is assumed to admit (a metric $d_X$ and) a Borel probability measure $\mu_X$.
We use $[X,d_X,\mu_X]$ (or $[X]$ when there is no confusion) to denote the mm-isomorphic class of $(X,d_X,\mu_X)$. 
Since a metric measure space $X$ is mm-isomorphic to the metric measure space $\supp \mu_X$ with the restricted metric of $d_X$ and the restricted measure of $\mu_X$, we also assume that $X=\supp \mu_X$ unless otherwise stated.
\end{convention} 

Let $\X$ denote the metric space of all mm-isomorphism classes of metric measure spaces spaces endowed with the box distance $\square$ (see \cite[Theorem 4.10]{shioya:16}), and for $n \in \N$, let $\Xeq{n}$ (resp., $\Xleq{n}$, $\Xfin$) denote the metric subspace of $\X$ consisting of all mm-isomorphism classes of metric measure spaces with precisely $n$ points (resp., at most $n$ points, finitely many points). 

\begin{remark} 
It is worth explicitly mentioning that $\GHcpt$ and $\X$ are sets because their objects are separable. 
Furthermore, the box distance always returns finite values (see Fact \ref{fact:trivial_upper_bound}), therefore, there is no need to restrict the family of objects considered. This situation is very different from the classical setting of the Gromov-Hausdorff distance. Indeed, the Gromov-Hausdorff distance between spaces that are not necessarily bounded can be infinite. Compactness is further required to show that two objects have Gromov-Hausdorff distance $0$ if and only if they are isometric.
\end{remark}

The box distance cannot be used to compare objects whose metrics are very different from each other. Indeed, the box distance is trivially bounded by $1$, which is obtained by considering the empty relation in \eqref{eq:box}. A more precise upper bound can be achieved.
\begin{fact}\label{fact:trivial_upper_bound}
Let $(X,d_X,\mu_X)$ and $(Y,d_Y,\mu_Y)$ be two metric measure spaces. Then,
$$\square(X,Y)\leq 1-\min\{\max_{x\in X}\mu_X(\{x\}),\max_{y\in Y}\mu_Y(\{y\})\}.$$
\end{fact}
\begin{proof}
Let $\overline x\in X$ and $\overline y\in Y$ be two points such that $\mu_X(\{\overline x\})=\max_x\mu_X(\{x\})$ and $\mu_Y(\{\overline y\})=\max_y\mu(\{y\})$. Consider $\pi^\prime$ the Borel measure on $X\times Y$ defined by
$$\pi^\prime(B)=\min\{\mu_X(\{\overline x\}),\mu_Y(\{\overline y\})\}\cdot\delta_{(\overline x,\overline y)}$$
for every Borel set $B$ of $X\times Y$, where $\delta_{(\overline x,\overline y)}$ indicates the probability measure of $X\times Y$ whose mass is concentrated on the point $(\overline x,\overline y)$. Since $\pi^\prime$ is trivially a {\em subtransport plan between $\mu_X$ and $\mu_Y$} (i.e., for every $A\subseteq X$ and $A^\prime\subseteq Y$, $\pi^\prime(A\times Y)\leq\mu_X(A))$ and $\pi^\prime(X\times A^\prime)\leq\mu_Y(A^\prime)$), there exists $\pi\in\Pi(\mu_X,\mu_Y)$ such that $\pi^\prime(B)\leq\pi(B)$ for every Borel subset $B$ of $X\times Y$ (see, for example, the proof of \cite[Proposition 4.12]{shioya:16}). Clearly, $\pi(\{(\overline x,\overline y)\})=\min\{\mu_X(\{\overline x\}),\mu_Y(\{\overline y\})\}$. By considering the relation $S=\{(\overline x,\overline y)\}$, the desired upper bound is attained.
\end{proof}

\begin{example}[The box distance captures only the small-scale geometry]\label{ex:box_distance_topological}
For every $t>0$, define $X_t=(\{0,1\},d_t,\mu_{\rm unif})$, where $d_t(0,1)=t$, and $\mu_{\rm unif}$ is the uniform probability measure. With abuse of notation, $X_0$ denotes the metric measure space with only one support point. Then,
$$\square(X_s,X_t)=\min\{\lvert t-s\rvert,1/2\}.$$
\begin{proof}
According to Fact \ref{fact:trivial_upper_bound}, $\square(X_s,X_t)\leq 1/2$. Furthermore, the  diagonal relation $\Delta=\{(0,0),(1,1)\}$ and the coupling $\pi$ defined as
$$\begin{pmatrix}
1/2 & 0\\
0 &1/2\end{pmatrix}
$$
shows that $\square(X_s,X_t)\leq\dis\Delta=\lvert t-s\rvert$.
If $\square(X_s,X_t)<1/2$, then it is easy to see that the relation $S$ realising the minimum in \eqref{eq:box} is a correspondence. Every correspondence contains either the diagonal $\Delta$ or the antidiagonal $\Delta^\prime=\{(0,1),(1,0)\}$ whose distortions are $\lvert t-s\rvert$. Therefore, $\dis S\geq\lvert t-s\rvert$.
\end{proof}
In particular, from some point on, the box distance fails to detect the difference between the two metric measure spaces. For example, $\square(X_1,X_t)=1/2$ for every $t\geq 3/2$.
\end{example}

We refer to \cite{kazukawa-nakajima-shioya:24} for a thorough discussion of topological properties of $\X$.

\subsection{Covering dimension}\label{sub:dim}
Let $X$ be a metric space and $\mathcal U$ be a family of subsets of $X$. We say that the {\em order of $\mathcal U$ is at most $n$} if, for every $x\in X$, there are at most $n+1$ distinct elements of $\mathcal U$ containing $x$. Another family $\mathcal V$ of subsets of $X$ is said to {\em refine $\mathcal U$} if, for every $V\in\mathcal V$, there is $U\in\mathcal U$ such that $V\subseteq U$.

\begin{definition}
Let $X$ be a metric space and $n\in\mathbb N\cup\{0\}$. Then, its {\em covering dimension is at most $n$}, and we write $\dim X\leq n$, if every open cover $\mathcal U$ of $X$ admits an open cover $\mathcal V$ that refines $\mathcal U$ and whose order is at most $n$. If no such $n$ exists, we say that the dimension of $X$ is infinite, and denote it by $\dim X=\infty$.
\end{definition}

Let us collect a few facts that will be used in the sequel. For the sake of simplicity, we do not state the results in full generality, and instead we refer to the monograph \cite{engelking:95}. 

\begin{fact}[{see \cite[Theorem 3.1.19]{engelking:95}}]
\label{fact:dim-subspace}
Let $f\colon X\to Y$ be a topological embedding between metric spaces. Then, $\dim X\leq\dim Y$. In particular, if $X$ and $Y$ are homeomorphic, $\dim X=\dim Y$.
\end{fact}

\begin{fact}[{see \cite[Theorem 1.8.3]{engelking:95}}]
\label{fact:dim-Rn-In}
For every $n\in\N$, $\dim\R^n=\dim[0,1]^n=n$.
\end{fact}

\begin{fact}[{see \cite[Theorem 3.1.8]{engelking:95}}]
\label{fact:countable-sum-thm}
If a metric space $X$ can be represented as the union of a sequence $\{F_k\}_{k\in\N}$ of closed subspaces such that $\dim F_k\leq n$ for any $k \in \N$, then $\dim X \leq n$.
\end{fact}

We will also use the following fact, which is obtained from \cite[Theorem 4.1]{nagami:60}.
\begin{fact}
\label{fact:nagami}
If $f\colon X\to Y$ is an open 
surjection
between metric spaces such that each fiber $f^{-1}(y)$, for $y \in Y$, is finite, then $\dim X =\dim Y$.
\end{fact}

\section{Dimension of Gromov-Hausdorff spaces and box metric spaces}\label{sec:dimensions}

For a set $A$, let $\card{A}$ denote its cardinality. 
For $k \in \mathbb{N}$ and for sets $A$ and $B$, let $[B]^k$ and $B^A$ denote the set of all subsets of $B$ with precisely $k$ points and that of all maps from $A$ to $B$, respectively. 
For $f \in B^A$, let $\graph f$ denote the graph $\{ (x,f(x)) \in A\times B \mid  x \in A\}$ of $f$.

For a metric space $X$, define its {\em separation} (see, for example, \cite{memoli:12}) as the value
$$\sep X=\min_{\{x,x'\} \in [X]^2}d_X(x,x^\prime).$$

\begin{lemma}
\label{lem:relation-bijection2}
Let $[X], [Y] \in  \Xfin$ and $\pi \in \Pi(\mu_X,\mu_{Y})$.
Suppose that $S \subset X \times Y$ satisfy $\dis S<\sep X$ and one of the following conditions holds:
\begin{enumerate}[{\rm (a)}]
\item $S$ is a correspondence, 
\item $\pi (S) > 1-\min_{z \in X} \mu_X(\{z\})$.
\end{enumerate}
Then, there exists an injection $f \colon X \to Y$ such that $\graph f\subseteq S$.
Moreover, if $\card{X}=\card{Y}$, then $f$ is bijective and $S=\graph f$.
\end{lemma}
\begin{proof}
For each $x \in X$, set $S_x =\{ y \in Y \mid (x,y)\in S\}$.
First, we show that $S_x \ne \emptyset$ for every $x \in X$.
If $S$ is a correspondence, then the claim obviously holds.
Suppose that $\pi (S) > 1-\min_{z \in X} \mu_X(\{z\})$ and let $x \in X$.
Since $\pi \in \Pi(\mu_X,\mu_Y)$, we have $\pi (\{x\}\times Y)= \mu_X(\{x\})$, and hence
\begin{align*}
\mu_X (\{x\}) &\geq \min_{z \in X} \mu_X(\{z\})
\\ &>1-\pi (S)= \pi ( (X\times Y) \setminus S ) 
 \geq \pi ( (\{x\}\times Y) \setminus (\{x\}\times S_x) ) 
\\ &=\pi (\{x\}\times Y)- \pi(\{x\}\times S_x) 
=\mu_{X}(\{x\}) - \pi (\{x\}\times S_x),
\end{align*}
which implies $\pi (\{x\}\times S_x)>0$, and thus $S_x\ne \emptyset$.

Then there exists a map $f \colon X \to Y$ such that $f (x) \in S_x$ for every $x \in X$.
For each $y \in Y$, since $\dis S < \sep X$, we have $\card{\{x \in X \mid (x,y)\in S\}}\leq 1$.
Hence, if $x,x' \in X$ and $x\ne x'$, then $S_x \cap S_{x'} = \emptyset$,
and thus $f$ is an injection. 
Moreover, if  $\card{X}=\card{Y}$, then $f$ is bijective and $S_x=\{f(x)\}$ for every $x \in X$, and thus $S=\graph f$.
\end{proof}

A subset $A$ of a topological space $X$ is said to be an {\em $F_\sigma$-set} of $X$ if $A$ is a countable union of closed subsets of $X$.

\begin{fact} \label{fact:GHX-closed-Fsigma}
Let $n \in \mathbb{N}$.
\begin{enumerate}[{\rm (1)}]
\item $\GHleq{n}$ is closed in $\GHfin$ and $\GHeq{n}$ is an $F_\sigma$-set of $\GHfin$.
\item $\Xleq{n}$ is closed in $\Xfin$ and $\Xeq{n}$ is an $F_\sigma$-set of $\Xfin$.
\end{enumerate}
\end{fact}
\begin{proof}
We give a proof of (2) for the sake of completeness. Item (1) can be proved similarly.

Let $[X] \in \Xfin \setminus \Xleq{n}$ and put 
\begin{align*}
\delta = \min\left\{ \sep X,\, \min_{z \in X} \mu_X (\{z\})\right \} >0.
\end{align*}
To show that 
$B_\square ([X],\delta) \cap \Xleq{n}=\emptyset$, let $[Y] \in B_\square ([X],\delta) \cap \Xfin$.
Since $\square ([X],[Y]) <\delta$, there exist $\pi \in \Pi(\mu_X,\mu_Y)$ and $S \subset X\times Y$ such that $\max\{ 1-\pi(S),\dis S\} < \delta$.
By Lemma \ref{lem:relation-bijection2}, there exists an injection $f \colon X \to Y$, and hence $\card{Y}\geq \card{X}>n$. Therefore $[Y]\notin \Xleq{n}$. 
Thus $B_\square ([X],\delta) \cap \Xleq{n}=\emptyset$, and $\GHleq{n}$ is closed in $\GHfin$.

Note that $\Xeq{n}=\Xleq{n} \setminus \Xleq{n-1}=\Xleq{n}\cap (\Xfin \setminus \Xleq{n-1})$.
Since $\Xfin \setminus \Xleq{n-1}$ open in the metric space $\Xfin$, it is an $F_\sigma$-set of $\Xfin$ (see \cite[Corollary 4.1.12]{engelking:89}).
Thus  $\Xeq{n}$ is an $F_\sigma$-set of $\Xfin$.
\end{proof}

\begin{lemma}
\label{lem:bijection-coupling-measure2}
Let $[X], [Y] \in \Xeq{n}$, $\pi \in \Pi(\mu_X,\mu_{Y})$ and $\delta>0$.
Suppose that there exists a bijection $f \colon X \to Y$ satisfying  $1-\pi (\graph f) <\delta$.
Then, $$\max_{x \in X} |\mu_X(\{x\}) -\mu_{Y} (\{f (x)\}) | <\delta.$$
\end{lemma}
\begin{proof}
Let $x \in X$. 
Since $\graph f \cap (X\times \{f(x)\}) = \{(x,f(x))\}$, 
we have
\begin{align*}
\mu_{Y}(\{f (x)\}) -\mu_{X}(\{x\}) 
&= \pi (X\times \{f (x)\}) - \pi (\{x\}\times Y)
\\ 
& \leq \pi ((X \times \{f (x)\}) \setminus \{ (x,f (x))\}  )
\\ 
& =\pi ((X \times \{f (x)\})\setminus \graph \sigma  )
\\ 
& \leq \pi ((X \times Y) \setminus \graph \sigma) =1-\pi (\graph \sigma) <\delta .
\end{align*}
Similarly, $\mu_{X}(\{x\}) - \mu_{Y}(\{f (x)\})<\delta$.
\end{proof}

For each $k\in \N$, 
we regard $k$ as the set $\{0,1,\dots, k-1\}$. 
Let  
\begin{align*}
R_n&=\{r\in (0,\infty)^{[n]^2}\mid 
\text{for any $\{i,j,k\} \in [n]^3$, }r(\{i,k\})\le r(\{i,j\})+r(\{j,k\})\}.
\end{align*}
Note that $\card{[n]^2}=\frac{n(n-1)}{2}$.
Let $\sint{\Delta_{n-1}}$ denote the interior of the $(n-1)$-dimensional simplex, which is regarded as the set of all $s \in (0,1)^n$ such that $\sum_{i\in n} s (i)=1$. 
We assume that $R_n$, $\sint{\Delta_{n-1}}$
and $R_n \times \sint{\Delta_{n-1}}$ are equipped with the supremum metrics.

Let $\sym{n}$ be the symmetric group on $n$. 
Then each $\sigma \in \sym{n}$ induces the map $\sigma\colon [n]^2 \to [n]^2; A \mapsto \sigma(A)$.
Define equivalence relations $\sim_{GH}$ on $R_n$ and $\sim_b$ on $R_n\times \sint{\Delta_{n-1}}$ as follows: 
\begin{itemize}
\item For $r,r' \in R_n$, 
$r\sim_{GH} r'$ if and only if there exists $\sigma \in \sym{n} $ such that $r' =r\circ \sigma$.
\item For $(r,s),(r',s') \in R_n$, $(r,s)\sim_b (r',s')$ if and only if there exists $\sigma \in \sym{n}$ such that $r' =r\circ \sigma$ and $s' =s\circ \sigma$.
\end{itemize}
For $r \in R_n$ (resp., $(r,s) \in R_n \times \sint{\Delta_{n-1}}$), let $[r]_{\sim_{GH}}$ (resp., $[r,s]_{\sim_b}$) denote the equivalence class of $r$ (resp., $(r,s)$) with respect to $\sim_{GH}$ (resp., $\sim_b$).
Note that $\card{[r]_{\sim_{GH}}}\leq n!$ and $\card{[r,s]_{\sim_b}} \leq n!$ for every $(r,s) \in R_n \times \sint{\Delta_{n-1}}$.

Let $R_n/{\sim_{GH}}$ and $(R_n\times \sint{\Delta_{n-1}})/{\sim_b}$ be the quotient spaces and $q_{GH} \colon R_n \to R_n/{\sim_{GH}}$ and  $q_b \colon R_n\times \sint{\Delta_{n-1}} \to (R_n\times\sint{\Delta_{n-1}})/{\sim_b}$ the natural quotient maps.

For $r\in R_n$, let $d_r$ denote the metric on $n$ defined by
\begin{align*}
d_r(i,j) = 
\begin{cases}
0  & \text{if } i=j,\\
r(\{i,j\}) & \text{if } i\ne j
\end{cases}
\end{align*}
for $i,j \in n$.
For $s \in \sint{\Delta_{n-1}}$, let $\mu_s$ denote the probability measure on $n$ defined by $\mu_s(A) = \sum_{i \in A}s(i)$ for $A \subset n$.
Then $\supp\mu_s =n$.
Define $\Phi_{GH}\colon R_n\to \GHeq{n}$ and $\Phi_b \colon R_n \times \sint{\Delta_{n-1}} \to \Xeq{n}$ by
\begin{align*}
\Phi_{GH}(r)&=[n,d_r] ,\quad  r \in R_n,
\\
\Phi_b(r,s) &=[n,d_r,\mu_s], \quad (r,s) \in R_n \times \sint{\Delta_{n-1}}.
\end{align*}

\begin{lemma}
\label{lem:Phigb}
The maps $\Phi_{GH}$ and $\Phi_b$ are Lipschitz open surjections such that $\Phi_{GH}^{-1}(\{\Phi_{GH}(r)\})=[r]_{\sim_{GH}}$ and $\Phi_b^{-1}(\{\Phi_b(r,s)\}) =[r,s]_{\sim_b}$ for every $(r,s) \in R_n \times \sint{\Delta_{n-1}}$.
\end{lemma}

\begin{proof}
To show that $\Phi_b$ is surjective, let $[X] \in \Xeq{n}$. 
Take a bijection $f\colon n \to \supp \mu_X$ and define $r \in (0,\infty)^{[n]^2} $ and $s \in (0,1)^n $ by $r(\{i,j\})=d_X(f(i),f(j))$ for $\{i,j\}\in [n]^2$ and $s(i) =\mu_X(\{f(i)\})$ for $i \in n$. Then $(r,s) \in R_n \times \sint{\Delta_{n-1}}$ and $\Phi_b(r,s)=[X]$.
Similarly, $\Phi_{GH}$ is surjective.

It is easy to see that $\Phi_{GH}^{-1}(\{\Phi_{GH}(r)\})=[r]_{\sim_{GH}}$ and $\Phi_b^{-1}(\{\Phi_b(r,s)\}) =[r,s]_{\sim_b}$ for every $(r,s) \in R_n \times \sint{\Delta_{n-1}}$.

To show that $\Phi_{GH}$ is Lipschitz, let $r, r' \in R_n$.
Let $D =\{(i,i) \mid i \in n \}$. Then $D$ is a correspondence between $(n,d_{r})$ and $(n,d_{r'})$ and 
\begin{align*}
d_{GH} (\Psi (r), \Psi(r')) \leq  \frac{1}{2}\dis D 
&= \frac{1}{2} \sup_{i,j \in n} |d_r(i,j)-d_{r'}(i,j)| 
\\&=\frac{1}{2} \sup_{\{i,j\} \in [n]^2} |r(\{i,j\})-r'(\{i,j\})|=\frac{1}{2}d_{R_n}(r,r'),
\end{align*}
where $d_{R_n}$ is the supremum metric on $R_n$. 
Hence $\Phi_{GH}$ is $\frac{1}{2}$-Lipschitz.
It follows from \cite[Lemma 4.24]{shioya:16} that $\Phi_b$ is $3n$-Lipschitz.

To show that $\Phi_b$ is open, let $O$ be an open subset of $R_n \times \sint{\Delta_{n-1}}$ and $[X]\in \Phi_b(O)$.
Take $(r_X,s_X) \in O$ such that $\Phi (r_X,s_X)=[X]$.
Let $d_\infty$ be the supremum metric on $R_n\times \sint{\Delta_{n-1}}$ and take $\delta>0$ such that $B_{d_\infty}((r_X,s_X),\delta) \subset O$ and $$\delta < \min\{ \min_{\{i,j\}\in[n]^2} r_X(\{i,j\}), \min_{i \in n} s_X(i) \}.$$

To show $B_\square ([X],\delta ) \cap \Xeq{n} \subset \Phi_b(O)$, let $[Y] \in B_\square ([X],\delta )\cap \Xeq{n}$.
Since $\Phi_b$ is surjective, there exists $(r_Y,s_Y) \in R_n \times \sint{\Delta_{n-1}}$ such that $\Phi_b(r_Y,s_Y)=[Y]$.
Since 
\begin{align*}
\square ([n,d_{r_X},\mu_{s_X}], [n,d_{r_Y},\mu_{s_Y}])=\square (\Phi_b(r_X,s_X),\Phi_b(r_Y,s_Y)) =\square ([X], [Y])<\delta,
\end{align*}
there exists $\pi \in \Pi(\mu_{s_X}, \mu_{s_Y})$ and $S \subset n \times n$ such that $\max\{1-\pi (S) , \dis S \} <\delta$.
Then, 
\begin{align*}
&\dis S < \delta < \min_{\{i,j\}\in[n]^2} r_X(\{i,j\})=
\sep(X,d_{r_X})  
\text{ and } 
\\
&\pi(S) >1-\delta > 1-\min_{i \in n} s_X(i)= 1-\min_{i \in n} \mu_{s_X}(\{i\}).
\end{align*}
Hence, by Lemma \ref{lem:relation-bijection2} there exists a bijection $\sigma \colon n \to n$ such that $S=\graph \sigma$. 
Then $\sigma \in \sym{n}$ and 
\begin{align*}
\max_{\{i,j\} \in [n]^2}|r_X(\{i,j\})-r_Y(\sigma(\{i,j\}))| 
&=\max_{\{i,j\} \in [n]^2}|d_{r_X} (i,j) -d_{r_Y} (\sigma(i),\sigma(j))|
\\&= \dis \graph \sigma <\delta.
\end{align*}
Since $1-\pi(\graph \sigma)<\delta$, by Lemma \ref{lem:bijection-coupling-measure2} we have
\begin{align*}
\max_{i\in n} |s_X(i)-s_Y(\sigma(i)) | = \max_{i \in n} |\mu_{s_X}(\{i\}) - \mu_{s_Y}(\{\sigma(i)\})|<\delta.
\end{align*}
Therefore $d_\infty((r_X,s_X),(r_Y\circ \sigma,s_Y\circ \sigma))<\delta$, which implies 
$$(r_Y\circ \sigma,s_Y\circ \sigma) \in B_{d_\infty}((r_X,s_X),\delta ) \subset O.$$
Since $(r_Y\circ \sigma,s_Y\circ \sigma) \in [r_Y,s_Y]_{\sim_b}=\Phi_b^{-1}(\Phi_b(r_Y,s_Y))$, we have
\begin{align*}
[Y]=\Phi_b (r_Y,s_Y)=\Phi_b(r_Y\circ \sigma,s_Y\circ \sigma) \in \Phi_b (O).
\end{align*}
Hence $B_\square ([X],\delta ) \cap \Xeq{n} \subset \Phi_b(O)$.

Thus $\Phi_b (O)$ is open in $\Xeq{n}$, and $\Phi_b$ is open.

Similarly, we can show that $\Phi_{GH}$ is open.
\end{proof}

\begin{corollary}
\label{cor:GHnXn-homeo}
The spaces $\GHeq{n}$ and $\Xeq{n}$ are homeomorphic to $R_n/{\sim_{GH}}$ and $(R_n \times \sint{\Delta_{n-1}})/{\sim_b}$, respectively.
\end{corollary}
\begin{proof}
By Lemma \ref{lem:Phigb}, $\Phi_{GH}$ and $\Phi_b$ are quotient maps (see \cite[Corollary 2.4.8]{engelking:89}) and induce homeomorphisms $\overline{\Psi}_{GH} \colon R_n/{\sim_{GH}} \to \GHeq{n}$  and $\overline{\Phi}_b \colon (R_n \times \sint{\Delta_{n-1}})/{\sim_b} \to \Xeq{n}$, respectively (see \cite[Proposition 2.4.3]{engelking:89}).
\end{proof}

\begin{lemma}
\label{lem:dimRDelta}
$\dim R_n=\frac{n(n-1)}{2}$ and $\dim (R_n\times \sint{\Delta_{n-1}})=\frac{(n+2)(n-1)}{2}$.
\end{lemma}

\begin{proof}
Since $[1,2]^{[n]^2} \subset R_n\subset \R^{[n]^2}$ and $\card{[n]^2}=\frac{n(n-1)}{2}$, by Facts \ref{fact:dim-subspace} and \ref{fact:dim-Rn-In} we have
\begin{align*}
\frac{n(n-1)}{2} =\dim [1,2]^{[n]^2}\leq \dim R_n \leq \R^{[n]^2}=\frac{n(n-1)}{2},
\end{align*} 
and hence  $\dim R_n=\frac{n(n-1)}{2}$.
Similarly, since $[0,1]^{\frac{(n+2)(n-1)}{2}}$ (homeomorphic to $[1,2]^{[n]^2}\times [0,1]^{n-1}$) can be topologically embedded into $R_n \times \sint{\Delta_{n-1}}$ 
and $R_n \times \sint{\Delta_{n-1}}$ can be topologically embedded into $\R^{\frac{(n+2)(n-1)}{2}}$ (homeomorphic to $\R^{[n]^2}\times \R^{n-1}$),
we have $\dim (R_n\times \sint{\Delta_{n-1}})=\frac{(n+2)(n-1)}{2}$.
\end{proof}

\begin{lemma}
\label{lem:dimRDelta/sim}
$\dim R_n/{\sim_{GH}}=\frac{n(n-1)}{2}$ and 
$\dim (R_n\times \sint{\Delta_{n-1}})/{\sim_b}=\frac{(n+2)(n-1)}{2}$.
\end{lemma}
\begin{proof}
For each $\sigma \in \sym{n}$, the self-maps $R_n\to R_n; r \mapsto r \circ \sigma$ and 
$R_n \times \sint{\Delta_{n-1}}; (r,s)\mapsto (r\circ \sigma , s \circ \sigma)$ are isometries, and we see that the quotient maps $q_{GH}\colon R_n \to R_n/{\sim_{GH}}$ and $q_b\colon  R_n\times \sint{\Delta_{n-1}}\to (R_n\times \sint{\Delta_{n-1}})/{\sim_b}$ are open. 
Note that each fiber of $q_{GH}$ and $q_b$ contains at most $n!$-many points.
Thus, according to Fact \ref{fact:nagami} and Lemma \ref{lem:dimRDelta},
we obtain the conclusion.
\end{proof}

By Corollary \ref{cor:GHnXn-homeo} and Lemma \ref{lem:dimRDelta/sim}, we obtain 
\begin{theorem}\label{thm:dimGHnXn}
$\dim \GHeq{n} = \frac{n(n-1)}{2}$ and $\dim \Xeq{n}=\frac{(n+2)(n-1)}{2}$.
\end{theorem}

From Theorem \ref{thm:dimGHnXn}, Theorem \ref{thm:dimension} stated in the Introduction follows.

\begin{proof}[Proof of Theorem \ref{thm:dimension}]
By Theorem \ref{thm:dimGHnXn} and Fact \ref{fact:dim-subspace} we have $\dim \GHleq{n} \geq \frac{n(n-1)}{2}$.
By Fact \ref{fact:GHX-closed-Fsigma} and Theorem \ref{thm:dimGHnXn}, $\GHleq{n}$ can be represented as a countable union of closed subspaces with dimension at most $\frac{n(n-1)}{2}$.
Hence, by Fact \ref{fact:countable-sum-thm}, we have $\dim \GHleq{n} \leq \frac{n(n-1)}{2}$. 
Thus $\dim \GHleq{n} = \frac{n(n-1)}{2}$.
Similarly, we see that  $\dim \Xleq{n}=\frac{(n+2)(n-1)}{2}$.
\end{proof}

A topological space $X$ is said to be \emph{strongly countable-dimensional} if $X$ can be represented as a countable union of finite-dimensional closed subspaces.
By Fact \ref{fact:GHX-closed-Fsigma} and Theorem \ref{thm:dimension}, we have the following.

\begin{corollary}\label{GHfinXfin-scd}
$\GHfin$ and $\Xfin$ are strongly countable-dimensional infinite dimensional spaces.
\end{corollary}

As a by-product of Corollary \ref{cor:GHnXn-homeo}, we have the following:
\begin{corollary}
\label{cor:MnembXn}
The space $\GHeq{n}$ can be topologically embedded into $\Xeq{n}$.
\end{corollary}
\begin{proof}
Let $s_b$ be the barycenter of $\sint{\Delta_{n-1}}$. Then the map $\varphi\colon R_n \to R_n \times \sint{\Delta_{n-1}}; r \mapsto (r,s_b)$ induces a continuous injection $\overline{\varphi}\colon R_n/{\sim_{GH}}\to (R_n \times \sint{\Delta_{n-1}})/{\sim_b}$ 
defined by $\overline{\varphi}(q_{GH}(r))=q_b(\varphi(r))$ for every $r\in R_n$. Let us represent the situation in the following diagram:
$$
\xymatrix{ R_n\ar^{\varphi}[rr]\ar_{q_{GH}}[d] & & R_n \times \sint{\Delta_{n-1}}\ar^{q_{b}}[d]\\
R_n/{\sim_{GH}}\ar^{\overline{\varphi}}[rr] & & (R_n \times \sint{\Delta_{n-1}})/{\sim_b}.
}
$$
Note that $\overline{\varphi}$ is well-defined. Indeed, if $r\sim_{GH}r'$, then  
$(r,s_b)\sim_{b}(r',s_b)$ since $s_b$ is the barycenter of $\sint{\Delta_{n-1}}$.

To show that $\overline{\varphi}\colon R_n/{\sim_{GH}}\to \overline{\varphi}(R_n/{\sim_{GH}})$ is open, let $U$ be an open subset of $R_n/{\sim_{GH}}$. Then $\varphi(q^{-1}_{GH}(U))$ is open in $\varphi(R_n)$ and there exists an open subset $V$ of $R_n \times \sint{\Delta_{n-1}}$ such that $V \cap \varphi(R_n) = \varphi(q^{-1}_{GH}(U))$. Then $q_b(V)$ is open in $(R_n \times \sint{\Delta_{n-1}})/{\sim_b}$ (see the proof of Lemma \ref{lem:dimRDelta/sim}).
Since $s_b$ is the barycenter, we also have $q_b(V)\cap \overline{\varphi}(R_n/{\sim_{GH}})=\overline{\varphi}(U)$.  
Hence  $\overline{\varphi}(U)$ is open in $\overline{\varphi} (R/{\sim_{GH}})$. 

Thus $\overline{\varphi}$ is a topological embedding, and the conclusion follows from Corollary \ref{cor:GHnXn-homeo}.
\end{proof}

\begin{question}
\label{que:MembX}
Can $\GHcpt$ (resp., $\GHfin$, $\GHleq{n}$) be topologically embedded into $\X$ (resp., $\Xfin$, $\Xleq{n}$)?
\end{question}

In the following remark, we briefly discuss how even simple cases of Question \ref{que:MembX} do not immediately follow from Corollary \ref{cor:MnembXn}.

\begin{remark}
The embedding $f_n\colon \GHeq{n}\to \Xeq{n}$ corresponding to $\overline{\varphi}$ in Corollary \ref{cor:MnembXn} is defined as $f([X,d_X])=[X,d_X,\mu_X]$, where $\mu_X$ is the uniform measure on $X$. Note the embedding $f_4\colon \GHeq{4}\to \Xeq{4}$ cannot be extended to a continuous map $\GHleq{4}\to \Xleq{4}$.
Indeed, let $X_k=\{-\frac{1}{k}, 0,1,3 \}$, $Y_k=\{0,1,3,3+\frac{1}{k}\}$, $k \in \N$, and $A=\{0,1,3\}$ be metric subspaces of $\R$. Then both sequences $\{[X_k]\}_k$ and $\{[Y_k]\}_k$ converge to $[A]$ in $\GHleq{4}$. But $\{f([X_k])\}_k$ converges to $[A,d_A,\mu_1]$, where $\mu_1(\{0\})=\frac{1}{2}$ and $\mu_1(\{1\})=\mu_1(\{3\})=\frac{1}{4}$, and $\{f([Y_k])\}_k$ converges $[A,d_A,\mu_2]$, where $\mu_2(\{0\})=\mu_2(\{1\})=\frac{1}{4}$ and $\mu_1(\{3\})=\frac{1}{2}$. Thus $\{f([X_k])\}_k$ and $\{f([Y_k])\}_k$ converge to different points in $\Xleq{4}$.
\end{remark}

\section{Proof of Theorem \ref{thm:HC-X}}\label{sec:proof-of-theoremB}

For the sake of convenience, we first give a direct proof of the following theorem, which is obtained from \cite[Theorem 1.3]{ishiki:22}.
Let $d$ denote the usual metric on $[0,1]$.

\begin{theorem}[{\cite{ishiki:22}}]
\label{thm:HC-GHcpt}
The Hilbert cube $[0,1]^\mathbb{N}$ can be topologically embedded into $\GHcpt$.
\end{theorem}
\begin{proof}
For $t \in [0,1]^\N$, let  
\begin{align}
\label{eq:Ct}
C_t=\{(0,0)\}\cup \bigcup_{i \in \N} (\{ 2^{-i+2}\} \times [0,2^{-i}(1+t(i))])
\end{align}
with the metric $d_t$ induced by the $\ell_1$-metric on $\R^2$ (see Figure \ref{fig:C_t} for a representation of the space).
Define $\iota_{GH} \colon [0,1]^\N \to \GHcpt$ by $\iota_{GH} (t) = [C_t,d_t]$ for $t \in [0,1]^\N$.

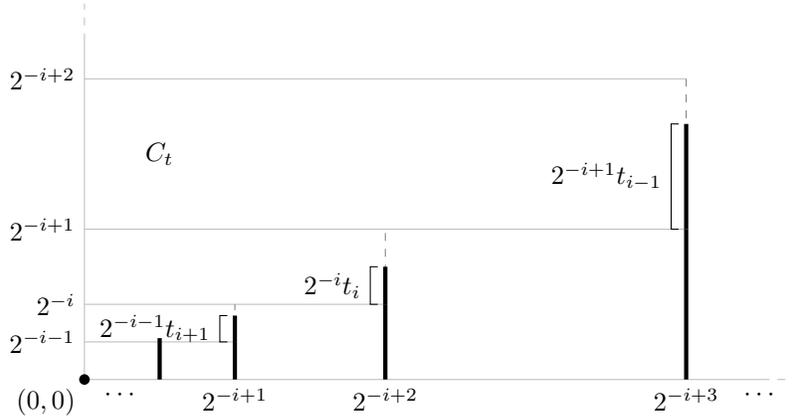
\begin{figure}[h!]
\centering
\begin{tikzpicture}
\draw[lightgray] (0,0)--(9,0) (0,0)--(0,4.5);
\draw[lightgray,dashed] (9,0)--(9.5,0) (0,4.5)--(0,5);
\draw[lightgray] (8,4)--(0,4);
\draw[lightgray] (8,2)--(0,2);
\draw[lightgray] (4,1)--(0,1);
\draw[lightgray] (2,0.5)--(0,0.5);
\draw (0,4) node[left]{$2^{-i+2}$};
\draw (0,2) node[left]{$2^{-i+1}$};
\draw (0,1) node[left]{$2^{-i}$};
\draw (0,0.5) node[left]{$2^{-i-1}$};
\fill (0,0) circle (2pt);
\draw (0,0) node[below left]{$(0,0)$};
\draw[gray,dashed] (2,0.5)--(2,1) (4,1)--(4,2) (8,2)--(8,4);
\draw[ultra thick] (2,0)--(2,0.5) (4,0)--(4,1) (8,0)--(8,2);
\draw (2,0) node[below]{$2^{-i+1}$};
\draw (4,0) node[below]{$2^{-i+2}$};
\draw (8,0) node[below]{$2^{-i+3}$};
\draw (1,3) node{$C_t$};
\draw[ultra thick] (2,0.5)--(2,0.85) (4,1)--(4,1.5) (8,2)--(8,3.4);
\draw (1.9,0.5)--(1.8,0.5)--(1.8,0.85)node[pos=0.5,left]{$2^{-i-1}t_{i+1}$}--(1.9,0.85);
\draw (3.9,1)--(3.8,1)--(3.8,1.5)node[pos=0.5,left]{$2^{-i}t_{i}$}--(3.9,1.5);
\draw (7.9,2)--(7.8,2)--(7.8,3.4)node[pos=0.5,left]{$2^{-i+1}t_{i-1}$}--(7.9,3.4);
\draw (0.5,0) node[below]{$\cdots$};
\draw (9,0) node[below]{$\cdots$};
\draw[ultra thick] (1,0)--(1,0.55);
\end{tikzpicture}
\caption{A representation of the space $C_t$ defined in the proof of Theorem \ref{thm:HC-GHcpt}. }\label{fig:C_t}
\end{figure}

To show that $\iota_{GH}$ is injective, let $s,t \in [0,1]^\N$ with $s\ne t$. 
Then there exists $i' \in \N$ such that $s(i')\ne t(i')$. 
We may assume $s(i')<t(i')$ without loss of generality. 
Then $(2^{-i'+2},2^{-i'}(1+t(i')) ) \in C_t$, 
$$d_t ((0,0), (2^{-i'+2},2^{-i'}(1+t(i')) ) ) =2^{-i'+2} + 2^{-i'}(1+t(i'))$$ and there is no point $x \in C_s$ satisfying $d_s ((0,0),x) = 2^{-i'+2} + 2^{-i'}(1+t(i'))$ 
(note that an isometry between $C_t$ and $C_s$ must associate the two copies of $(0,0)$). 
Hence $C_s$ and $C_t$ are not isometric and $\iota_{GH}(s)\ne \iota_{GH}(t)$. Thus $\iota_{GH}$ is injective.

To show that $\iota_{GH}$ is continuous, let $t \in [0,1]^\N$ and $\varepsilon>0$. 
Choose $i_0 \in \N$ such that $2^{-i_0}<\varepsilon$. 
Put  $$U=\prod_{i=1}^{i_0} B_d(t(i),\varepsilon) \times \prod_{i=i_0+1}^{\infty}[0,1].$$
Then $U$ is a neighborhood of $t$ in $[0,1]^\N$. 
Let $s \in U$ and let $d_H$ be the Hausdorff distance with respect to the $\ell_1$-metric on $\R^2$. 
Then we see that $$d_{GH}(\iota_{GH}(s),\iota_{GH}(t)) \leq d_{H} (C_s,C_t)<\varepsilon,$$ and hence $\iota_{GH}$ is continuous. 

Since $[0,1]^\N$ is compact, $\iota_{GH}$ is a topological embedding.
\end{proof}

\begin{proof}[Proof of Theorem \ref{thm:HC-X}]
For $t \in [0,1]^\N$, let $(C_t,d_t)$ be the metric space defined in \eqref{eq:Ct}.
We define a Borel probability measure $\mu_t$ on $C_t$ as follows: Let $\nu$ denote the Lebesgue measure on $[0,1]$.
Let $i \in \mathbb{N}$ and define a Borel measure $\mu_{t,i}$ on $[0,2^{-i}(1+t(i))]$ by 
\begin{align*}
\mu_{t,i} (A) = \frac{\nu(A)}{1+t(i)} 
\end{align*}
for each Borel subset $A$ of $[0,2^{-i}(1+t(i))]$. Then $\mu_{t,i} ([0,2^{-i}(1+t(i))]) =2^{-i}$.
For a Borel subset $B$ of $C_t$, let 
\begin{align}
\label{eq:HC-X:Bi}
B_i=\{x \in [0,2^{-i}(1+t(i))] \mid (2^{-i+2},x) \in B\}.
\end{align}
Define a Borel measure $\mu_t$ on $C_t$ by 
$$\mu_t(B) =\sum_{i=1}^\infty \mu_{t,i}(B_i)$$
for each Borel subset $B$ of $C_t$.  Then $\nu_t(C_t)=1$ and $\supp\nu_t =C_t$.

Define $\iota_{b} \colon [0,1]^\N \to \X$ by $\iota_{b} (t) = [C_t,d_t,\mu_t]$ for $t \in [0,1]^\N$.
By the same reason as in the proof of Theorem \ref{thm:HC-GHcpt}, $\iota_{b}$ is injective.
Therefore it suffices to show that $\iota_b$ is continuous.
Let $t \in [0,1]^\N$ and $\varepsilon>0$.
Choose $i_*, m \in \N$ such that 
$2^{-i_*}<\frac{\varepsilon}{4}$ and $\frac{1}{m}<\frac{\varepsilon}{4}$,
and define a neighborhood $V$ of $x$ by 
$V=\prod_{i=1}^{i_*} B_d(t(i),\frac{\varepsilon}{4}) \times \prod_{i=i_*+1}^{\infty}[0,1]$. 
Let $s \in V$. 
Note that 
\begin{align}
\label{eq:HC-H:|s(i)-t(i)|}
\forall i \in \N,\, 2^i|s(i)-t(i)|<\frac{\varepsilon}{4}.
\end{align}
We show that $\square ((C_t,d_t,\mu_t), (C_s,d_s,\mu_s))\leq \varepsilon$.

Fix $i \in \N$, $u \in \{s,t\}$ and $j \in m$, and put
\begin{align*}
I_{u,i,j}=  \left [\frac{2^{-i}(1+u(i))j}{m}, \frac{2^{-i}(1+u(i))(j+1)}{m} \right  ].
\end{align*}
Then, 
\begin{align}
\label{eq:HC-X:diamI}
\diam_d I_{u,i,j}=\frac{2^{-i}(1+u(i))}{m}\leq \frac{1}{m}
<\frac{\varepsilon}{4}
\end{align}
and 
\begin{align}
\label{eq:HC-X:muI}
\mu_{u,i}(I_{u,i,j}) = 
\frac{1}{2^im}.
\end{align}

Let $\pi_{i,j}$ be the restriction of the product measure of $\mu_{s,i}$ and $\mu_{t,i}$ to $I_{s,i,j} \times I_{t,i,j}$.
For a Borel subsets $A$ of $C_s$ (resp., $A'$ of $C_t$), 
$A_i\cap I_{s,i,j}$ (resp., $A'_i\cap I_{t,i,j}$) is a Borel subset of $I_{s,i,j}$ (resp., $I_{t,i,j}$), where $A_i$ (resp., $A'_i$) is the subset of $[0,2^{-i}(1+s(i))]$ (resp., $[0,2^{-i}(1+t(i))]$) defined in \eqref{eq:HC-X:Bi}, and we have
\begin{align*}
&\pi_{i,j} ( (A_i \cap I_{s,i,j}) \times  I_{t,i,j})=\frac{\mu_{s,i}(A_i \cap I_{s,i,j})}{2^im},
\\
&\pi_{i,j} ( I_{s,i,j} \times (A'_i \cap I_{t,i,j}) )=\frac{\mu_{t,i}(A'_i \cap I_{t,i,j})}{2^im}.
\end{align*}

Unfix $i \in \N$, $u\in \{s,t\}$ and $j \in m$.
For a Borel subset $B$ of $C_s\times C_t$, $i \in \N$ and $j \in m$, let
\begin{align*}
B_{i,j}= \{(x,y)\in I_{s,i,j}\times I_{t,i,j} \mid ((2^{-i+2}, x), (2^{-i+2},y)) \in B  \},
\end{align*}
and define a Borel measure $\pi$ on $C_s \times C_t$ by 
\begin{align*}
\pi (B) = \sum_{i\in \N}\sum_{j\in m}2^i m\, \pi_{i,j} (B_{i,j})
\end{align*}
for each Borel subset $B$ of $C_s \times C_t$.
Then $\pi$ is a coupling between $\mu_s$ and $\mu_t$.
Indeed, for a Borel subset $A$ of $C_s$, we have
\begin{align*}
\pi (A \times C_t) 
&= \sum_{i\in \N }\sum_{j\in m}2^i m\, \pi_{i,j} ((A \times C_t)_{i,j})
\\ 
&= \sum_{i\in \N }\sum_{j\in m}2^i m\, \pi_{i,j} ((A_i \cap I_{s,i,j}) \times  I_{t,i,j})
\\
&= \sum_{i\in \N }\sum_{j\in m}\mu_{s,i}(A_i \cap I_{s,i,j})
= \sum_{i\in \N} \mu_{s,i}(A_i ) = \mu_s(A),
\end{align*}
and similarly, $\pi(C_s\times A') =\mu_t(A')$ for every Borel subset $A'$ of $C_t$.

Let $2^{-\infty +2} =0$, $s(\infty)=t(\infty)=0$, $m_\infty=0$ and $I_{s,\infty,0}=I_{t,\infty,0}$ for the sake of convenience, and put 
\begin{align}\label{eq:HC-X:S}
\begin{split}
S&=\bigcup_{i \in \N\cup \{\infty\}}\bigcup_{j\in m}(\{2^{-i+2}\} \times I_{s,i,j}) \times (\{2^{-i+2}\} \times I_{t,i,j})
\\&=\{((0,0),(0,0))\}\cup \bigcup_{i \in \N}\bigcup_{j\in m} (\{2^{-i+2}\} \times I_{s,i,j}) \times (\{2^{-i+2}\} \times I_{t,i,j}).
\end{split}
\end{align}
Then $S$ is a closed subset of $C_s\times C_t$ and, by \eqref{eq:HC-X:muI}, 
\begin{align*}
\pi(S) 
&= \sum_{i\in \N }\sum_{j\in m}2^i m\, \pi_{i,j} (S_{i,j})
= \sum_{i\in \N }\sum_{j\in m}2^i m\, \pi_{i,j} (I_{s,i,j}\times I_{t,i,j})
\\
& =\sum_{i\in \N }\sum_{j\in m}2^i m\, \mu_{s,i,j}(I_{s,i,j}) \mu_{t,i,j} ( I_{t,i,j})  
= \sum_{i\in \N }\sum_{j\in m}\frac{1}{2^i m} 
=1.
\end{align*}
It suffices to show that $\dis S\leq \varepsilon$.
For this purpose, let $(x_0,y_0), (x_1,y_1) \in S$.
Then, by \eqref{eq:HC-X:diamI} and \eqref{eq:HC-X:S}, for $k \in 2$ there exist $i_k \in \N\cup \{\infty\}$, $j_k \in m_{i_k}+1$ and $a_k, b_k \in [0,\frac{\varepsilon}{4})$ such that
\begin{align*}
x_k &=\left (2^{-i_k+2}, \frac{2^{-i_k}(1+s(i_k))j_k}{m}+  a_k \right) \text{ and } 
\\
y_k &=\left(2^{-i_k+2},\frac{2^{-i_k}(1+t(i_k))j_k}{m}+b_k\right).
\end{align*}
Then 
\begin{align*}
d_s(x_0,x_1) &= |2^{-i_0+2}-2^{-i_1+2}| + \left|\frac{2^{-i_0}(1+s(i_0))j_0}{m}+a_0-\left(\frac{2^{-i_1}(1+s(i_1))j_1}{m}+a_1\right)\right|, 
\\
d_t(y_0,y_1) &= |2^{-i_0+2}-2^{-i_1+2}| + \left|\frac{2^{-i_0}(1+t(i_0))j_0}{m}+b_0-\left(\frac{2^{-i_1}(1+t(i_1))j_1}{m}+b_1\right)\right|.
\end{align*}
Since $||a-a'|-|b-b'||\leq |a-b|+|a'-b'|$ for any $a,a',b,b' \in \R$, by \eqref{eq:HC-H:|s(i)-t(i)|} we have
\begin{align*}
&\, |d_s(x_0,x_1)-d_t(y_0,y_1)| 
\\
\leq &\,\left|\frac{2^{-i_0}(1+s(i_0))j_0}{m}+a_0 - \left(\frac{2^{-i_0}(1+t(i_0))j_0}{m}+b_0\right) \right|
\\
& \qquad +
\left|\frac{2^{-i_1}(1+s(i_1))j_1}{m}+a_1 - \left(\frac{2^{-i_1}(1+t(i_1))j_1}{m}+b_1\right) \right|
\\
\leq 
& \,2^{-i_0}|s(i_0)-t(i_0)|\frac{j_0}{m}  +|a_0-b_0| +2^{-i_1}|s(i_1)-t(i_1)|\frac{j_1}{m}  +|a_1-b_1|
\\< &\, \frac{\varepsilon}{4}+\frac{\varepsilon}{4}+\frac{\varepsilon}{4}+\frac{\varepsilon}{4} =\varepsilon.
\end{align*}
Therefore $\dis S\leq \varepsilon$. Thus $\iota_b$ is a topological embedding.
\end{proof}


\begin{thebibliography}{99}
\bibitem{assouad:83} 
A. Assouad, 
{\it Plongements lipschitziens dans $\R^n$}, 
Bull. Soc. Math. France 111 (1983), 429--448.	
	
\bibitem{burago-burago-ivanov:01}
D. Burago, Y. Burago,  S. Ivanov,
{\it A course in metric geometry},
American Mathematical Society, Providence, RI, 2001. 

\bibitem{edwards:75} 
D. A. Edwards, 
{\em The structure of superspace}, 
published in: Studies in Topology, Academic Press, 1975.

\bibitem{engelking:89}
R. Engelking, 
{\it General topology}, Heldermann Verlag, 1989.

\bibitem{engelking:95}
R. Engelking, 
{\it Theory of dimensions finite and infinite}, Heldermann Verlag, 1995.

\bibitem{gromov:81} 
M. Gromov, 
{\em Structures M\'etriques pour les Vari\'et\'es Riemanniennes} (J. Lafontaine and P. Pansu, eds.), 
Textes Math\'ematiques, 1, CEDIC, Paris, 1981.

\bibitem{gromov:93} 
M. Gromov, 
{\em Asymptotic invariants of infinite groups}, 
in ``Geometric group theory, Vol. 2'' (Sussex, 1991), 1--295, London Math. Soc. Lecture Note Ser., 182, Cambridge Univ. Press, Cambridge, 1993.

\bibitem{iliadis-ivanov-tuzhilin:17} 
S. Iliadis, A. O. Ivanov, A. A. Tuzhilin, 
{\em Local structure of Gromov-Hausdorff space, and isometric embeddings of finite metric spaces into this space}, 
Topology Appl, 221 (2017), 393--398

\bibitem{ishiki:22}
Y. Ishiki,
{\it Branching geodesics of the Gromov-Hausdorff distance}, 
Anal. Geom. Metr. Spaces 10 (2022), no. 1, 109--128.

\bibitem{kadets:75}
M. I. Kadets, 
{\em Note on the gap between subspaces},
Funkts. Anal. Prilozhen. 9 (1975), no. 2, 73--74; English transl. in Funct. Anal. Appl. 9 (1975), 156--157.

\bibitem{kazukawa-nakajima-shioya:24}
D. Kazukawa, H. Nakajima, T. Shioya, 
{\it Topological aspects of the space of metric measure spaces},
Geom. Dedicata 218 (2024), no. 3, Paper No. 68, 28 pp.

\bibitem{loehr:13}
W. L\"ohr, 
{\em Equivalence of Gromov-Prohorov- and Gromov's $\underline\square_\lambda$-metric on the space of metric measure spaces}, 
Electron. Commun. Probab. 18 (2013), no. 17, 10.

\bibitem{memoli:11}
F. M\'emoli, 
{\it Gromov-Wasserstein distances and the metric approach to object matching}, 
Found. Comput. Math. 11, 417?487 (2011).

\bibitem{memoli:12}
F. M\'emoli, 
{\it Some properties of Gromov-Hausdorff distances}, 
Discrete \& Computational Geometry 48 (2) (2012), 416--440.


\bibitem{nagami:60}
K. Nagami, 
{\it Mappings of finite order and dimension theory},
Jpn. J. Math. 30 (1960), 25--54. 

\bibitem{nakajima:22}
H. Nakajima,
{\it Box distance and observable distance via optimal transport},
preprint, arXiv:2204.04893v2.

\bibitem{petersen:06}
P. Petersen, 
{\it Riemannian geometry}, 
2nd edition, Graduate Text in Mathematics, Springer (2006).

\bibitem{pritchard-weighill:24} 
N. Pritchard, T. Weighill, 
{\em Coarse embeddability of Wasserstein space and the space of persistence diagrams}, 
preprint, arXiv:2307.12884.

\bibitem{ostrovska-ostrovskii:25}
S. Ostrovska, M.I. Ostrovskii,
{\it Universality and non-embeddability into Banach spaces of subspaces of the real line with the Gromov-Hausdorff distance},
preprint, https://facpub.stjohns.edu/ostrovsm/OstrovGH.pdf.

\bibitem{shioya:16}
T. Shioya, 
{\it Metric measure geometry}, EMS Publishing House, Z\"{u}rich, 2016.

\bibitem{tuzhilin:20} 
A. A. Tuzhilin, 
{\it Lectures on Hausdorff and Gromov-Hausdorff distance geometry},  arXiv:2012.00756v1. 

\bibitem{villani:03}
C. Villani, 
{\it Topics in optimal transportation},
Graduate Studies in Mathematics, vol. 58, American Mathematical Society, 2003.

\bibitem{zava:24}
N. Zava,
{\it Coarse and bi-Lipschitz embeddability of subspaces of the Gromov-Hausdorff space into Hilbert spaces}, 
to appear in Algebr. Geom. Topol., 
arXiv:2303.04730.

\end{thebibliography}
\end{document}